\documentclass{amsart}

\usepackage{mathrsfs,amssymb,multirow,setspace,color}
\usepackage{hyperref}
\usepackage[a4paper, total={7in, 9in}]{geometry}
\theoremstyle{plain}
\usepackage{graphicx}
\newtheorem{theorem}{Theorem}[section]
\newtheorem{lemma}[theorem]{Lemma}
\newtheorem{proposition}[theorem]{Proposition}
\newtheorem{corollary}[theorem]{Corollary}

\theoremstyle{definition}

\theoremstyle{remark}
\newtheorem{remark}[theorem]{Remark}

\input xy
\xyoption{all}

\def\d{\mathcal{D}(G)}

\def\pgh{\mathcal{P}_E({G})}

\begin{document}
\title{Metric Dimension of Difference Graph of Finite Groups }

\author[ Manisha, Parveen, Jitender Kumar]{ Manisha, Parveen, Jitender Kumar$^{*}$}
 \address{$\text{}^1$Department of Mathematics, Birla Institute of Technology and Science Pilani, Pilani-333031, India}
  \address{$\text{}^2$Department of Mathematics, Indian Institute of Technology Guwahati, India}
\email{ yadavmanisha2611@gmail.com, p.parveenkumar144@gmail.com, jitenderarora09@gmail.com}

\begin{abstract} 
The Difference graph $\d$ of a finite group $G$ is the difference of the enhanced power graph $\pgh$ and the power graph $\mathcal{P}(G)$ with all the isolated vertices removed. In this paper, we characterize the vertex set of the difference graph of finite nilpotent groups and obtain its cardinality. Consequently, we obtain the metric dimension of the difference graph of finite nilpotent groups. Moreover, this paper determines the metric dimension of the difference graphs of certain non-nilpotent groups, namely: dihedral groups, the generalized quaternion groups, and the semi-dihedral groups.
 \end{abstract}


\keywords{ edge connectivity, vertex connectivity, minimum degree, nilpotent groups. \\ *  Corresponding author}

\maketitle
\section{Historical Background and Preliminaries}

Several graphs can be defined on finite groups, such as Cayley graphs, power graphs, commuting graphs, and their variants. These graphs help in understanding the connection between algebraic properties of groups and graph theoretic properties of these graphs. Using group theory often leads to useful constructions and results in graph theory.

The \emph{power graph} $\mathcal{P}(G)$ of a group $G$ is a simple undirected graph with vertex set $G$. Two distinct vertices $x$ and $y$ are adjacent if one is a power of the other, that is, $x \in \langle y \rangle$ or $y \in \langle x \rangle$. The directed power graph was introduced by Kelarev and Quinn \cite{powerintroduced}. Later, the undirected power graph was studied from different viewpoints, see \cite{powerisomorphism,powermetric,powerminimum}. A survey on power graphs can be found in \cite{powersurvey}.

To study graphs lying between the power graph and the commuting graph, Aalipour et al. \cite{enhancedintro} introduced the \emph{enhanced power graph}. The enhanced power graph $\mathcal{P}_E(G)$ has vertex set $G$, where two distinct vertices $x$ and $y$ are adjacent if both belong to some cyclic subgroup of $G$. This graph has been studied in \cite{enhanedconnectivity,enhancedmetric,enhancedisomorphism}.

Cameron \cite{graphdefined} posed several open problems on graphs defined on groups, including the difference between the enhanced power graph and the power graph. The \emph{difference graph} $\mathcal{D}(G)$ is obtained by removing all edges of the power graph from the enhanced power graph and deleting isolated vertices. Biswas et al. \cite{differencebiswas} studied the connectivity and perfectness of $\mathcal{D}(G)$. Later, Parveen et al. \cite{differencesplit,parveen2023nilpotent} studied structural and topological properties of difference graphs of finite nilpotent groups.

For a graph $\Gamma$, the distance between vertices $x$ and $y$ is denoted by $d(x,y)$. A vertex $z$ resolves $x$ and $y$ if $d(x,z) \neq d(y,z)$. A subset $W \subseteq V(\Gamma)$ is called a resolving set if every pair of vertices is resolved by some vertex in $W$. The \emph{metric dimension} of $\Gamma$, denoted by $\dim(\Gamma)$, is the minimum size of a resolving set.
Metric dimension was introduced independently by Harary and Melter \cite{1970metric} and Slater \cite{slaterleaves}. Determining the metric dimension is NP-hard \cite{NPhard,np}. It has applications in navigation, network discovery, and robotics. The metric dimension of power graphs and enhanced power graphs was studied in \cite{powermetric,enhancedmetric}. Motivated by this, we study the metric dimension of difference graphs. Our results show that the metric dimension of $\mathcal{D}(G)$ does not follow directly from those of $\mathcal{P}(G)$ and $\mathcal{P}_E(G)$.

We now recall some basic definitions and notation used throughout the paper. A graph $\Gamma = (V,E)$ consists of a vertex set $V(\Gamma)$ and an edge set $E(\Gamma)$. Two vertices $u$ and $v$ are adjacent, written as $u \sim v$, if $\{u,v\} \in E(\Gamma)$. All graphs considered here are simple.
The open neighbourhood of a vertex $x$ is $N(x)$, and the closed neighbourhood is $N[x]=N(x)\cup\{x\}$. A graph is connected if there exists a path between every pair of vertices. The degree of a vertex $u$ is denoted by $\deg(u)$.
Two vertices $u$ and $v$ are called \emph{true twins} if $N[u]=N[v]$, and \emph{false twins} if $N(u)=N(v)$. If they are either, then they are called \emph{twins}. A set $U$ is called a twin set if every pair of vertices in $U$ are twins.

\begin{remark}\label{twinsets}\rm\cite[Remark 3.3]{twinset}
	Let $U$ be a twin set in a connected graph $\Gamma$ with $|U|=l\ge 2$. Then every resolving set of $\Gamma$ contains at least $l-1$ vertices of $U$.
\end{remark}
We define an equivalence relation on $G$ by $x \sim y$ if $x$ and $y$ are twins. Let $[x]$ denote the equivalence class of $x$. For an element $a \in G$, the \emph{order} of $a$, denoted by $o(a)$, is the size of the cyclic subgroup generated by $a$. If $|G|=p^n$ for some prime $p$, then $G$ is called a \emph{$p$-group}. The set of all element orders of $G$ is denoted by $\pi_G$.

A cyclic subgroup of $G$ is called \emph{maximal cyclic} if it is not properly contained in any cyclic subgroup of $G$. The set of all maximal cyclic subgroups of $G$ is denoted by $\mathcal{M}(G)$. Let $M(G)$ be the set of all generators of maximal cyclic subgroups of a finite group $G$.

For $n \ge 3$, the \emph{dihedral group} $D_{2n}$ of order $2n$ is defined by
\[
D_{2n}=\langle x,y \mid x^n=y^2=e,\; xy=yx^{-1}\rangle.
\]
It has one maximal cyclic subgroup $\langle x\rangle$ of order $n$ and $n$ maximal cyclic subgroups $\langle x^iy\rangle$, $1\le i\le n$, each of order $2$.

For $n \ge 2$, the \emph{generalized quaternion group} $Q_{4n}$ of order $4n$ is defined by
\[
Q_{4n}=\langle a,b \mid a^{2n}=e,\; a^n=b^2,\; ab=ba^{-1}\rangle.
\]
It has one maximal cyclic subgroup $\langle a\rangle$ of order $2n$ and $n$ maximal cyclic subgroups $\langle a^ib\rangle$, $1\le i\le n$, each of order $4$.

For $n \ge 2$, the \emph{semidihedral group} $SD_{8n}$ of order $8n$ is defined by
\[
SD_{8n}=\langle a,b \mid a^{4n}=b^2=e,\; ba=a^{2n-1}b\rangle.
\]
This group has one maximal cyclic subgroup $\langle a\rangle$ of order $4n$, $2n$ maximal cyclic subgroups of order $2$, and $n$ maximal cyclic subgroups of order $4$.

Moreover,
\[
\mathcal{D}(D_{2n}) \cong \mathcal{D}(\mathbb{Z}_n), \quad
\mathcal{D}(Q_{4n}) \cong \mathcal{D}(\mathbb{Z}_{2n}), \quad
\mathcal{D}(SD_{8n}) \cong \mathcal{D}(\mathbb{Z}_{4n}).
\]

\begin{theorem}[{\rm \cite[p.193]{b.dummit1991abstract}}]\label{nilpotent}
	Let $G$ be a finite group. The following are equivalent:
	\begin{enumerate}
		\item $G$ is nilpotent.
		\item Every Sylow subgroup of $G$ is normal.
		\item $G$ is the direct product of its Sylow subgroups.
		\item Elements of coprime orders commute.
	\end{enumerate}
\end{theorem}

\begin{theorem}[{\rm\cite[Theorem 5.4]{parveenenhanced}}]\label{enhanced strong product}
	If $G=P_1\times P_2\times \cdots \times P_r$ is a nilpotent group, then
	\[
	\mathcal{P}_E(G) \cong \mathcal{P}_E(P_1)\boxtimes \mathcal{P}_E(P_2)\boxtimes \cdots \boxtimes \mathcal{P}_E(P_r),
	\]
	where $P_i$ is the Sylow $p_i$-subgroup of $G$.
\end{theorem}

\begin{proposition}[{\rm\cite[Proposition 3.4]{differencesplit}}]\label{difference adjacency}
	Let $G$ be a finite group and let $x,y \in G\setminus\{e\}$ with $o(x)\mid o(y)$. Then $x$ is not adjacent to $y$ in $\mathcal{D}(G)$.
\end{proposition}

\begin{lemma}[{\rm\cite{cameron2025order}}]
	Let $p$ be a prime and $G=\mathbb{Z}_{p^{r_1}}\times\cdots\times\mathbb{Z}_{p^{r_k}}$, where $r_1\ge r_2\ge\cdots\ge r_k$.
	\begin{enumerate}
		\item $G$ has exactly $\frac{p^k-1}{p-1}$ elements of order $p$.
		\item Every non-maximal cyclic subgroup of order $p^u$, $u>1$, is contained in exactly $p^{k-1}$ cyclic subgroups of order $p^{u+1}$.
	\end{enumerate}
\end{lemma}

\begin{remark}
	If $x$ and $y$ are elements of a finite group $G$ such that neither $o(x)\mid o(y)$ nor $o(y)\mid o(x)$, then $x \nsim y$ in $\mathcal{P}(G)$. The converse holds if $x$ and $y$ lie in the same cyclic subgroup.
\end{remark}

The notation $x \twoheadrightarrow y$ means that $y \in \langle x\rangle$ but $x \notin \langle y\rangle$.

 \section{Main Results}

\begin{proposition}\label{adjacency in difference graph}
Let $G= P_1 \times P_2\times \cdots \times P_k$ be a finite nilpotent group. Then two vertices
\[
x=(x_1,x_2,\ldots,x_k)\quad \text{and}\quad y=(y_1,y_2,\ldots,y_k)
\]
are adjacent in the difference graph $\mathcal{D}(G)$ if and only if there exist distinct indices $i,j \in \{1,2,\ldots,k\}$ such that
\[
x_i \twoheadrightarrow y_i \quad \text{and} \quad y_j \twoheadrightarrow x_j,
\]
and for all $l \notin \{i,j\}$, the elements $x_l$ and $y_l$ belong to the same cyclic subgroup of $P_l$.
\end{proposition}

\begin{proof}
Suppose that $x=(x_1,x_2,\ldots,x_k)$ and $y=(y_1,y_2,\ldots,y_k)$ are adjacent in $\mathcal{D}(G)$. By Theorem~\ref{enhanced strong product} for each $i \in \{1,2,\ldots,k\}$, the vertices $x_i$ and $y_i$ lie in the same cyclic subgroup of $P_i$. Since $x$ and $y$ are adjacent in $\mathcal{D}(G)$ and so they are not adjacent in the power graph $\mathcal{P}(G)$. By Proposition~\ref{difference adjacency}, this implies that neither $o(x)\mid o(y)$ nor $o(y)\mid o(x)$. Write
\[
o(x)=\prod_{i=1}^k o(x_i) \quad \text{and} \quad o(y)=\prod_{i=1}^k o(y_i),
\]
we conclude that there must exist at least two distinct indices $i$ and $j$ such that
\[
o(x_i) > o(y_i) \quad \text{and} \quad o(x_j) < o(y_j).
\]
Consequently, we have $x_i \twoheadrightarrow y_i$ and $y_j \twoheadrightarrow x_j$. For all remaining indices $l \notin \{i,j\}$, the elements $x_l$ and $y_l$ belong to the same cyclic subgroup of $P_l$, as required.

Conversely, suppose that there exist distinct indices $i,j$ such that $x_i \twoheadrightarrow y_i$ and $y_j \twoheadrightarrow x_j$, and that $x_l$ and $y_l$ lie in the same cyclic subgroup of $P_l$ for all $l \notin \{i,j\}$. By Theorem~\ref{enhanced strong product}, $x$ and $y$ are adjacent in the enhanced power graph $\mathcal{P}_E(G)$. Moreover, the conditions $x_i \twoheadrightarrow y_i$ and $y_j \twoheadrightarrow x_j$ imply that neither $o(x)\mid o(y)$ nor $o(y)\mid o(x)$. Hence, $x$ and $y$ are not adjacent in the power graph $\mathcal{P}(G)$. Therefore, $x$ and $y$ are adjacent in the difference graph $\mathcal{D}(G)$.
\end{proof}

\begin{proposition}\label{isolated vertices in nilpotent group}
Let $G= P_1 \times P_2\times \cdots \times P_k$
be a finite nilpotent group. Then a vertex
$
x=(x_1,x_2,\ldots,x_k)
$
does not belong to the vertex set of the difference graph $\mathcal{D}(G)$ if and only if either
\(
x_i \in M(P_i) \text{ for all } i \in [k],
\)
or
\(
x_i = e \text{ for all } i \in [k].
\)
Moreover,
\(
|V(\mathcal{D}(G))| = |G|-\big(|M(P_1)|\,|M(P_2)|\cdots |M(P_k)|+1\big).
\)
\end{proposition}

\begin{proof}
Suppose that 
$
x=(x_1,x_2,\ldots,x_k)
$
is not a vertex of $\mathcal{D}(G)$. Then $x$ is an isolated vertex in the graph $\mathcal{P}_E(G)- \mathcal{P}(G)$. Assume, that there exists at least one index $i$, say $i=1$, such that $x_1 \notin M(P_1)$.

\medskip
\noindent\textbf{Case 1:} $x_1 = e$.

\noindent\emph{Subcase 1.1:} $x_j = e$ for all $2 \le j \le k$.  
Then $x=(e,e,\ldots,e)$, and the claim follows.

\noindent\emph{Subcase 1.2:} $x_j \neq e$ for some $j \neq 1$, say $j=2$.  
Choose
\[
y=(g_1,e,x_3,\ldots,x_k),
\]
where $g_1 \in M(P_1)$. Then $x$ and $y$ both lie in the cyclic subgroup
\[
\langle (g_1,x_2,x_3,\ldots,x_k)\rangle,
\]
and hence are adjacent in $\mathcal{P}_E(G)$. Moreover, neither $o(x)\mid o(y)$ nor $o(y)\mid o(x)$ holds. Therefore, $x$ and $y$ are adjacent in $\mathcal{D}(G)$, contradicting the assumption that $x$ is isolated.

\medskip
\noindent\textbf{Case 2:} $x_1 \neq e$.

\noindent\emph{Subcase 2.1:} $x_j \neq e$ for some $j \neq 1$, say $j=2$.  
Choose
\[
y=(g_1,e,x_3,\ldots,x_k),
\]
where $g_1 \in M(P_1)$. As before, $x,y \in \langle (g_1,x_2,\ldots,x_k)\rangle$ and so $x\sim y$ in $\mathcal{P}_E(G)$. Since neither $o(x)\mid o(y)$ nor $o(y)\mid o(x)$, we obtain $x\sim y$ in $\mathcal{D}(G)$, again a contradiction.

\noindent\emph{Subcase 2.2:} $x_j = e$ for all $2 \le j \le k$.  
Choose
\[
y=(e,g_2,e,\ldots,e),
\]
where $g_2 \in M(P_2)$. Then $x,y \in \langle (x_1,g_2,e,\ldots,e)\rangle$, and neither $o(x)\mid o(y)$ nor $o(y)\mid o(x)$. Hence, $x$ and $y$ are adjacent in $\mathcal{D}(G)$, a contradiction.

\medskip
Thus, if $x$ is not a vertex of $\mathcal{D}(G)$, then either $x_i \in M(P_i)$ for all $i \in [k]$, or $x_i = e$ for all $i \in [k]$.

Conversely, suppose that either $x_i \in M(P_i)$ or $x=e$ for all $i \in [k]$. Then for any $y \in G$ such that $x$ and $y$ are adjacent in $\mathcal{P}_E(G)$, we have either $o(x)\mid o(y)$ or $o(y)\mid o(x)$. Hence, $x$ has no neighbors in $\mathcal{D}(G)$ and so it is an isolated vertex of $\mathcal{P}_E(G)-\mathcal{P}(G)$.

 The number of isolated vertices is $|M(P_1)|\,|M(P_2)|\cdots |M(P_k)| + 1$, which yields the stated formula for $|V(\mathcal{D}(G))|$.
\end{proof}
Define an equivalence relation on $G$ by declaring $x \approx y$ if and only if $x$ and $y$ are twins in the graph $\mathcal{D}(G)$. The equivalence class containing $x$ is denoted by $[x]$.

\begin{lemma}\label{equivalence class}
Let $G= P_1 \times P_2\times \cdots \times P_k$ be a finite nilpotent group and let $x \in P_1$, where $x \neq e$ and $o(x)=p_1^{\alpha}$. Then
\[
[(x,e,e,\ldots,e)]
=
\Big\{(y,e,e,\ldots,e)\;:\; o(y)=p_1^{\alpha}
\text{ and }
|\langle y\rangle \cap \langle x\rangle| \ge p_1^{\alpha-1}
\Big\}.
\]
\end{lemma}

\begin{proof}
Let $(x,e,e,\ldots,e)$ and $(x_1,x_2,\ldots,x_k)$ be twins in the difference graph $\mathcal{D}(G)$, that is,
\[
N(x,e,e,\ldots,e)=N(x_1,x_2,\ldots,x_k).
\]

Suppose that $x_i \neq e$ for some $i>1$. Without loss of generality, assume $x_2 \neq e$. Consider the vertex
\[
(e,x_2,e,\ldots,e).
\]
Since $x \twoheadrightarrow e$ and $x_2 \twoheadrightarrow e$, we have
\[
(e,x_2,e,\ldots,e)\in N(x,e,e,\ldots,e).
\]
However,
\[
o(e,x_2,e,\ldots,e)\mid o(x_1,x_2,\ldots,x_k),
\]
and hence $(e,x_2,e,\ldots,e)$ and $(x_1,x_2,\ldots,x_k)$ are not adjacent in $\mathcal{D}(G)$. This contradicts our assumption. Therefore, $x_i=e \quad \text{for all } i\ge 2.$ Thus,
\[
N(x,e,e,\ldots,e)=N(x_1,e,e,\ldots,e).
\]

Now let $(z_1,z_2,\ldots,z_k)\in N(x,e,e,\ldots,e)$. By Proposition~\ref{adjacency in difference graph}, we obtain $x \twoheadrightarrow z_1 \quad \text{and} \quad x_1 \twoheadrightarrow z_1.$ Thus, both $\langle x\rangle$ and $\langle x_1\rangle$ contain $\langle z_1\rangle$. This implies that $o(x_1)=o(x)=p_1^{\alpha},$ and either $\langle x_1\rangle=\langle x\rangle$ or
\[
|\langle x_1\rangle \cap \langle x\rangle| \ge p_1^{\alpha-1}.
\]


This completes the proof.
\end{proof}
\begin{lemma}\label{unique involution}
Let $G= P_1 \times P_2\times \cdots \times P_k$ be a finite nilpotent group. Then $G$ has a unique element $l$ of order $2$ if and only if $|[l]|=1$ in $\mathcal{D}(G)$.
\end{lemma}

\begin{proof}
 If $G$ has a unique element $l$ of order $2$, then there is no other vertex in $\mathcal{D}(G)$ having the same neighborhood as $l$. Thus, $l$ has no twin and therefore $|[l]|=1$.

Conversely, assume that $|[l]|=1$. If $o(l)>2$, then $l\neq l^{-1}$ and in $\mathcal{D}(G)$ we have $N(l)=N(l^{-1})$. Therefore, $l$ and $l^{-1}$ are twins, which implies $|[l]|\ge 2$, a contradiction. Therefore $o(l)=2$.\\
Now suppose that $G$ contains more than one element of order $2$. Then all the elements of order $2$ are in same equivalence class. This implies that $|[l]|\ge 2$, again a contradiction. Thus, $l$ is the unique element of order $2$ in $G$.
\end{proof}


\begin{lemma}\label{relation of elements order}
Let $x=(x_1,\ldots,x_k)$ and $y=(y_1,\ldots,y_k)$ be two vertices of the difference graph $\mathcal{D}(G)$ of a finite nilpotent group $G$ such that $N(x)=N(y)\cup\{z\}$ for some $z \in V(\mathcal{D}(G))$. Then 
\begin{enumerate}
    \item[(i)] $z$ is a unique element of order $2$.
    \item[(ii)] $o(y)=2o(x)$.
\end{enumerate}

\end{lemma}

\begin{proof}
Suppose that $N(x)=N(y)\cup\{z\}$. Clearly, $z\in N(x)$ but $z\notin N(y)$. If $o(z)>2$, then $z^{-1}\neq z$ and $N(z)=N(z^{-1})$, which implies $z^{-1}\in N(x)$. This contradicts the assumption that $N(x)\setminus N(y)$ contains exactly one vertex. Hence $o(z)=2$.

Next, we show that $G$ has a unique element of order $2$. Suppose, on the contrary, that $G$ contains more than one element of order $2$. Since $x\sim z$, we have $2\nmid o(x)$. Then every element of order $2$ is adjacent to $x$, implying $|N(x)\setminus N(y)|\ge 2$, a contradiction. Hence $z$ is the unique element of order $2$ in $G$.

Now let $p$ be an odd prime. We claim that $p^{a}\mid o(y)$ if and only if $p^{a}\mid o(x)$. First assume that $p^{a}\nmid o(x)$. Choose an element $l=(l_1,l_2,\ldots,e)$ such that $o(l_1)=2$, $o(l_2)=p^{a}$, and $x_2\in\langle l_2\rangle$. Then neither $o(l)\mid o(x)$ nor $o(x)\mid o(l)$. Also, $l,x$ lie in a common cyclic subgroup. Hence $l\sim x$. Since $2p^{a}\mid o(y)$, we have $l\not\sim y$, contradicting $N(x)=N(y)\cup\{z\}$.

Now we suppose that $p^{a}\nmid o(y)$. Choose $l=(e,l_2,\ldots,e)$ with $o(l_2)=p^{a}$ and $y_2\in\langle l_2\rangle$. Then neither $o(l)\mid o(y)$ nor $o(y)\mid o(l)$ and $l\sim y$. Since $p^{a}\mid o(x)$, we have $o(l)\mid o(x)$, which implies $l\not\sim x$, again a contradiction. Thus $p^{a}\mid o(y)$ if and only if $p^{a}\mid o(x)$ for all odd primes $p$.

Finally, suppose that $2^{a}\mid o(y)$ for some $a>1$. Then an element of order $2^{a}$ is not adjacent to $x$. Otherwise, $N(x) \neq N(y) \cup \{z\}$. Consequently, $2\mid o(x)$. This contradicts the fact that $x\sim z$ where $o(z)=2$. Therefore $2\mid o(y)$ and $2\nmid o(x)$.

Combining all cases, we conclude that $o(y)=2o(x)$.
\end{proof}


\begin{lemma}\label{nbh condition}
Let $x=(x_1,\ldots,x_k)$ and $y=(y_1,\ldots,y_k)$ be two elements of a finite nilpotent group $G$. Then $N(x)=N(y)\cup\{z\}$ in $\mathcal{D}(G)$ if and only if  either $G \cong \mathbb{Z}_{2^n}\times G',n >1 \text{ or } G \cong Q_{4m} \times G', m =2^r, r \geq 1$, where $G'$ is a Sylow $p$-subgroup with a maximal cyclic subgroup of prime order.
\end{lemma}

\begin{proof}
Suppose $N(x)=N(y)\cup\{z\}$. By Lemma~\ref{relation of elements order}, $G$ has a unique element of order $2$ and $2\nmid o(x)$. The only finite nilpotent groups with a unique element of order $2$ are $\mathbb{Z}_{2^n}\times G'$ with $n\ge 1$ and $Q_{4m}\times G'$ with $m=2^r$, $r\ge 1$, where $G'$ is a group of odd order.

Assume that $|G'|$ is divisible by at least two primes.
Since $2\nmid o(x)$, 
 without loss of generality consider $p_i\mid o(x)$.

\textbf{Case 1:} $p_j\nmid o(x), j \neq i$. Then $x=(e,x_i,\ldots,e)$ with $o(x_i)=p_i^{a_i}$. Since $p_i^{a_i}\mid o(x)$, we have $p_i^{a_i}\mid o(y)$. Choose $s=(e,y_i,e,\ldots,s_j,\ldots,e)$ with $o(s_j)=p_j$. Then $o(s)=p_i^{a_i}p_j$ and $o(y)=2p_i^{a_i}$. Thus neither $o(y)\mid o(s)$ nor $o(s)\mid o(y)$ holds, and $s,y$ lie in a common cyclic subgroup, implying $s\sim y$. However, $o(x)\mid o(s)$ implies $x\nsim s$, a contradiction.

\textbf{Case 2:} $p_j\mid o(x), j \neq i$. Choose $u$ such that $o(u)=2\prod_{r\neq 1,j} o(x_r)$. Then $o(u)\mid o(y)$ implies $u\nsim y$. Also, $o(u)\nmid o(x)$ and $o(x)\nmid o(u)$, while $u$ and $x$ lie in a common cyclic subgroup, implying $u\sim x$, a contradiction.

Hence $|G|$ is divisible by at most two primes.

Now assume $G'$ does not contain a maximal cyclic subgroup of prime order. Then $o(x)=p^{\alpha}$ and $o(y)=2p^{\alpha}$.

\textbf{Case 1:} $\alpha=1$. Choose $s=(e,s_2)$ with $o(s_2)=p^2$ and $y_2\in\langle s_2\rangle$. Then $o(x)\mid o(s)$ implies $x\nsim s$, while $o(y)\nmid o(s)$ and $o(s)\nmid o(y)$ imply $s\sim y$, a contradiction.

\textbf{Case 2:} $\alpha>1$. Choose $s=(y_1,s_1)$ with $s_1\in\langle x_1\rangle$ and $o(s_1)=p$. Then $o(s)\mid o(y)$ implies $y\nsim s$, while $o(x)\nmid o(s)$ and $o(s)\nmid o(x)$ imply $x\sim s$, again a contradiction.

Thus $G'$ contains a maximal cyclic subgroup of prime order.

Conversely, suppose $G\cong \mathbb{Z}_{2^n}\times G'$ with $n>1$ or $G\cong Q_{4m}\times G'$ with $m=2^r$, $r\ge 1$, where $G'$ is a Sylow $p$-subgroup with a maximal cyclic subgroup of prime order. Choose $x=(e,x_2)$ with $o(x_2)=p$ and $y=(y_1,y_2)$ with $o(y_1)=2$ and $o(y_2)=p$. Then $N(x)=N(y)\cup\{z\}$, where $z$ is the unique element of order $2$.
\end{proof}

\begin{theorem}\label{resolvingset}
Let $G= P_1 \times P_2\times \cdots \times P_k$ be a finite nilpotent group and let $x_1,\ldots,x_s$ be a system of representatives of the equivalence classes in $\mathcal{D}(G)$. If $G \cong \mathbb{Z}_{2^n} \times G'$ with $n>1$, or $G \cong Q_{4m} \times G'$ with $m=2^r$, $r \geq 1$, where $G'$ is a Sylow $p$-subgroup having a maximal cyclic subgroup of prime exponent, then
\[
S = (V(\mathcal{D}(G))  \setminus \{x_1,\ldots,x_s\})\cup \{x_k\},
\]
where $o(x_k)=2$, is a resolving set of $\mathcal{D}(G)$. Otherwise,
\[
S = V(\mathcal{D}(G)) \setminus \{x_1,\ldots,x_s\}
\]
is a resolving set of $\mathcal{D}(G)$. Moreover,
\[
\dim(\mathcal{D}(G)) \in \{\alpha,\alpha-1\},
\]
where $\alpha = |G| - \big(|M(P_1)||M(P_2)|\cdots|M(P_k)| + |\overline{G}|\big)$.
\end{theorem}
\begin{proof}
Let $x_i \neq x_j$ be two distinct elements of $\{x_1,\ldots,x_s\}$. Then there exists $l_1 \in V(\mathcal{D}(G))$ such that $l_1 \in N(x_i)$ but $l_1 \notin N(x_j)$.

\textbf{Case 1.} \emph{$|[l_1]|\neq 1$.}  
Since $[l_1]$ is a non-singleton equivalence class, there exists a vertex  
$t \in S = V(\mathcal{D}(G)) \setminus \{x_1,\ldots,x_s\}$ such that $t \in [l_1]$.  
Therefore $x_i \sim t$ and $x_j \nsim t$, which implies $d(x_i,t)=1$ and $d(x_j,t)\neq 1$. Hence $t$ resolves $x_i$ and $x_j$.

\textbf{Case 2.} \emph{$|[l_1]|=1$.}  
In this case, $l_1=x_k$ and hence $N(x_i)=N(x_j)\cup\{x_k\}$. By Lemma~\ref{nbh condition}, either  
$G \cong \mathbb{Z}_{2^n}\times G'$ with $n>1$, or  
$G \cong Q_{4m}\times G'$ with $m=2^r$, $r\geq 1$,  
where $G'$ is a Sylow $p$-subgroup having a maximal cyclic subgroup of prime exponent.

Assume $x_i=(e,x)$, where $x$ is a generator of the maximal cyclic subgroup of prime exponent, and  
$x_j=(y_1,x)$ with $o(y_1)=2$. Then $N(x_i)=N(x_j)\cup\{x_k\}$, and $x_k$ is the unique vertex that distinguishes $x_i$ and $x_j$. Since $x_k \in \{x_1,\ldots,x_s\}$, we include $x_k$ and define  
\[
S = (V(\mathcal{D}(G)) \setminus \{x_1,\ldots,x_s\} )\cup \{x_k\}.
\]
Thus $x_k \in S$, and $S$ resolves $x_i$ and $x_j$.

Therefore, in both cases, the set $S$ resolves every pair of distinct equivalence class representatives. The result now follows from Remark~\ref{twinsets} and Proposition~\ref{isolated vertices in nilpotent group}.
\end{proof}

\begin{lemma}\label{cyclic_equiv_classes}
Let $G$ be a finite cyclic group of order $n$, where $n$ is not a prime power. Then the number of equivalence classes of $\mathcal{D}(G)$ under the relation $\sim$ is $\tau(n)-2$, where $\tau(n)$ denotes the number of positive divisors of $n$.
\end{lemma}

\begin{proof}
Let $x,y \in G$ such that $[x]=[y]$. If $o(x)\neq o(y)$ and $o(x)=p^{\alpha}$ for some prime $p$, then by Lemma~\ref{equivalence class} we must have $o(y)=p^{\alpha}$, a contradiction. Hence assume that $o(x)$ is divisible by at least two distinct primes.

\textbf{Case 1.} \emph{$o(x)\nmid o(y)$.}  
Since $G$ is cyclic, $x$ and $y$ lie in a common cyclic subgroup, and hence $x \sim y$. Without loss of generality assume $y \neq y^{-1}$. Then $x \sim y^{-1}$ but $y \nsim y^{-1}$, contradicting $[x]=[y]$.

\textbf{Case 2.} \emph{$o(x)\mid o(y)$.}  
Let $o(x)=p_1^{\alpha_1}\cdots p_k^{\alpha_k}$ and $o(y)=p_1^{\beta_1}\cdots p_k^{\beta_k}$ with $\alpha_i \leq \beta_i$ for all $i$ and $\alpha_j < \beta_j$ for some $j$, where $k\ge 2$. Choose an element $y'$ with $o(y')=p_j^{\beta_j}$. Then $x \sim y'$ but $y \nsim y'$, contradicting $[x]=[y]$.

Therefore $o(x)=o(y)$, and hence equivalence classes correspond exactly to element orders. Since the identity and generators form singleton classes, the total number of equivalence classes is $\tau(n)-2$.
\end{proof}
\begin{corollary}\label{metric dimension of Zn}
    Let $G$ be a finite cyclic group of order $n$ and $n$ is not of prime power order. Then $dim(\mathcal{D}(G))$ is $n-\phi(n)-\tau(n)+2$ if $G \cong \mathbb{Z}_{2^k}\times \mathbb{Z}_{p}, p \neq 2, k >1$ otherwise $dim(\mathcal{D}(G)) = n -\phi(n)- \tau(n)+1$.
\end{corollary}
\begin{proof}
    Every finite cyclic group of order $n$ is isomorphic to $\mathbb{Z}_{n}$. Therefore the number of equivalence class of $\mathbb{Z}_{n}$ are $\tau(n)-2$. By Theorem \ref{resolvingset}, $dim(\mathcal{D}(\mathbb{Z}_n)) = n -\phi(n) - \tau(n)+2$ for $G \cong \mathbb{Z}_{2^k} \times \mathbb{Z}_{p}$ otherwise $dim{(\mathcal{D}(\mathbb{Z}_n))} = n - \phi(n)- \tau(n)+1$.
\end{proof}

\begin{theorem}\label{dihedral_dim}
For the Dihedral group $D_{2n}= \langle x, y  :  x^{n} = y^2 = e,  xy = yx^{-1} \rangle $, if the cyclic subgroup $\langle x\rangle$ of order $n$ is isomorphic to $\mathbb{Z}_{2^k}\times \mathbb{Z}_p$, where $p\neq 2$ and $k>1$, then
\[
\dim(\mathcal{D}(D_{2n})) = n-\phi(n)-\tau(n)+2.
\]
Otherwise,
\[
\dim(\mathcal{D}(D_{2n})) = n-\phi(n)-\tau(n)+1.
\]
\end{theorem}

\begin{proof}
Since $\mathcal{D}(D_{2n}) \cong \mathcal{D}(\mathbb{Z}_n)$, the number of equivalence classes in $\mathcal{D}(D_{2n})$ is $\tau(n)-2$ by Lemma~\ref{cyclic_equiv_classes}.  

By Theorem~\ref{resolvingset}, we obtain
\[
\dim(\mathcal{D}(D_{2n})) =
\begin{cases}
n-\phi(n)-\tau(n)+2, & \text{if } \mathbb{Z}_n \cong \mathbb{Z}_{2^k}\times \mathbb{Z}_p, \\
n-\phi(n)-\tau(n)+1, & \text{otherwise}.
\end{cases}
\]
\end{proof}

\begin{theorem}\label{quaternion_dim}
For the generalized quaternion group ${Q}_{4n} = \langle a, b  :  a^{2n}  = e, a^n= b^2, ab = ba^{-1} \rangle$, if the cyclic subgroup $\langle a \rangle $ is isomorphic to $\mathbb{Z}_{2^k}\times \mathbb{Z}_p$, where $p\neq 2$ and $k>1$, then
\[
\dim(\mathcal{D}(Q_{4n})) = 2n-\phi(2n)-\tau(2n)+2.
\]
Otherwise,
\[
\dim(\mathcal{D}(Q_{4n})) = 2n-\phi(2n)-\tau(2n)+1.
\]
\end{theorem}
\begin{theorem}\label{semidihedral_dim}
For the semi-dihedral group  $SD_{8n} = \langle a, b  :  a^{4n} = b^2 = e,  ba = a^{2n -1}b \rangle$, if the cyclic subgroup $\langle a \rangle $ is isomorphic to $\mathbb{Z}_{2^k}\times \mathbb{Z}_p$, where $p\neq 2$ and $k>1$, then
\[
\dim(\mathcal{D}(SD_{8n})) = 4n-\phi(4n)-\tau(4n)+2.
\]
Otherwise,
\[
\dim(\mathcal{D}(SD_{8n})) = 4n-\phi(4n)-\tau(4n)+1.
\]
\end{theorem}

\section*{Declarations}

\textbf{Funding}: The first and the third author wishes to acknowledge the support of Core Research Grant (CRG/2022/001142) funded by  ANRF. The second author gratefully acknowledges the financial support received from the Indian Institute of Technology Guwahati under the Institute Post-Doctoral Fellowship (IITG/R\&D/IPDF/2024-25/20240828P1181).

\textbf{Conflicts of interest/Competing interests}: There is no conflict of interest regarding the publishing of this paper.

\textbf{Availability of data and material (data transparency)}: Not applicable.

\vspace{.3cm}
\textbf{Code availability (software application or custom code)}: Not applicable. 
 
\bibliographystyle{abbrv}


\vspace{1cm}
 \noindent
 {\bf Manisha\textsuperscript{\normalfont 1}, \bf Parveen\textsuperscript{\normalfont 1}, \bf Jitender Kumar\textsuperscript{\normalfont 1}}
 \bigskip

 \noindent{\bf Addresses}:

 \vspace{5pt}

\end{document}